\documentclass[11pt]{amsart}

\usepackage{amsfonts}
\usepackage{amssymb}
\usepackage{graphicx,color}
\usepackage{epstopdf}
\usepackage{epsfig}

\usepackage{amsmath}

\usepackage{amsthm}

\title[Probability, Minimax Approximation]{\bf
Probability, Minimax Approximation and Nash-Equilibrium. Estimating the Parameter of a Biased Coin}

\newcommand{\be}{\begin{equation}}
\newcommand{\ee}{\end{equation}}

\newcommand{\supp}{\mathop{\rm supp}}

\newtheorem{theorem}{Theorem}
\newtheorem{lemma}[theorem]{Lemma}
\newtheorem{corollary}[theorem]{Corollary}

\theoremstyle{remark}
\newtheorem{remark}[theorem]{Remark}
\newtheorem{example}[theorem]{Example}
\theoremstyle{definition}
\newtheorem{definition}[theorem]{Definition}

\begin{document}

\centerline{\bf  \Large }
\author[D. Benko]{D. Benko}
\address{Department of Mathematics and Statistics, University of South Alabama, Mobile, AL 36688, USA}
\email{dbenko@southalabama.edu}

\author[D. Coroian]{D. Coroian }
\address{Department of Mathematical Sciences, Indiana University-Purdue University, Fort Wayne, IN 46805, USA }
\email{coroiand@ipfw.edu}

\author[P. Dragnev]{P. D. Dragnev $^{\dagger}$}
\address{Department of Mathematical Sciences, Indiana University-Purdue University, Fort Wayne, IN 46805, USA }
\email{dragnevp@ipfw.edu}
\thanks{\noindent $^{\dagger}$ The research of this author was supported, in part, by a Simons Foundation grant no. 282207.}

\author[R. Orive]{R. Orive $^{*}$}
\address{Departmento de An\'{a}lisis Matem\'{a}tico, Universidad de La Laguna, 38200, The Canary Islands, Spain }
\email{rorive@ull.es}
\thanks{\noindent $^{*}$ The research of this author was supported, in part, by Ministerio de Ciencia e Innovaci\'{o}n under grant MTM2015-71352-P,
and was conducted while visiting IPFW as a Scholar-in-Residence.}

\vskip 1 cm

\begin{abstract} This paper deals with the application of Approximation Theory type techniques to study a classical problem in Probability: estimating the parameter of a biased coin. For this purpose, a Minimax Estimation problem is considered and the characterization of the optimal estimator is shown, together with the weak asymptotics of such optimal choices as the number of coin tosses approaches infinity; in addition, a number of numerical examples and graphs are displayed. At the same time, the problem is also discussed from the Game Theory viewpoint, as a non-cooperative, two--player game, and a Nash-equilibrium is established. The particular case of $n=2$ tosses is completely solved.

\end{abstract}


\maketitle

{\em `` Dedicated to Prof. Walter Gautschi on the occasion of his 90-th birthday.''}

\vskip 1 cm

\setcounter{equation}{0}
\setcounter{theorem}{0}
\section{Introduction}
\label{intro}

The following problem is well-known in the area of Probability. A biased coin is given, but the probability of heads is unknown. We flip it $n$ times and get $k$ heads. The problem is to estimate the probability of heads.
The most typical approach to solve this problem is the Maximum Likelihood method (see e.g. \cite{LC} or \cite{P}). Let $p=P(heads)$. Then

$$P(k \ {\rm heads\ out\ of}\ n\ {\rm tosses})= {n \choose k } p^k(1-p)^{n-k}.$$

Since we have absolutely no information about $p$, we choose an estimator $\widehat{p}\in [0,1]$ for which this expression is maximal, that is, $\widehat{p} = k/n$.
This approach has several shortcomings.
1.) For small $n$ we get unrealistic estimations. For example, if $n=1$ and we get a head, the method gives the estimation $p=1$, and if we get a tail, the method gives $p=0$.
2.) the method yields the most likely value of $p$, but does not take into account the error in the estimation. This can be seen in the following example: suppose we flip the coin $n=4$ times and get $k=2$ heads; of course, $\widehat{p} = .5$. However, while in the literature a coin is often considered ``reasonably fair'' when $.45 \leq p \leq .55$, in this case it is easy to check that $P\left((|p-.5|>.05)\ \vert \ (n=4, k=2)\right) = .8137...\,$, assuming $p$ is chosen following a uniform distribution, and, thus, the probability of dealing with a biased coin is high despite the Maximum Likelihood estimator.

Another approach for estimating the probability $p$ of a biased coin deals with Bayesian or Minimax Estimation. While in the former the goal is minimizing the average risk, in the latter the aim is to minimize the maximum risk; nevertheless, there are other well-known methods, such as the so-called Uniform Minimum Variance Unbiased (UMVU) estimator (for more information, see \cite{LC}). These methods require the use of a \textit{loss function}, which measures the difference between the parameter and its estimator. In the current paper, the loss function $|p-a_i|$ will be considered, where $a_i$ is the estimation of $p$ if $i$ heads are obtained after $n$ tosses, for $i=0,\ldots,n\,$. Then, the expected value of this loss function, commonly called \textit{risk} or \textit{penalty function}, is given by
\begin{equation}
\label{penalty}
D(a_0,...,a_n;p):=\sum_{k=0}^n {n \choose k} p^k(1-p)^{n-k} |p-a_k|
\end{equation}
and our goal will be to choose $a_i$'s that minimize the sup-norm of $D$. Figure \ref{FiveToss} below, corresponding to the case of $n=5$ tosses, illustrates the motivation for our analysis. It shows the penalty functions for the Maximum Likelihood choice (i.e. $a_k = k/n\,,\,k=0,\ldots,n$) and for the choice of the $a_i$'s we will prove to be optimal (see Section 2). Indeed, the sup-norm of the penalty function $D$ for the optimal choice is clearly smaller than the one corresponding to the Maximum Likelihood approach. Moreover, Figure \ref{FiveToss} shows that the optimal choice satisfies what we call hereafter the ``equimax'' property. This is very similar to the characterization of the set of interpolation nodes to reach the optimal Lebesgue function, as conjectured by Bernstein and Erd\H{o}s and proved forty years later by Kilgore \cite{K} and de Boor-Pinkus \cite{BP}. This similarity served as an important motivation to apply Approximation Theory type techniques for investigating this problem, and will be discussed in more detail in Section 2.

\begin{figure}[h]
\centering\includegraphics*[width=2.5in]{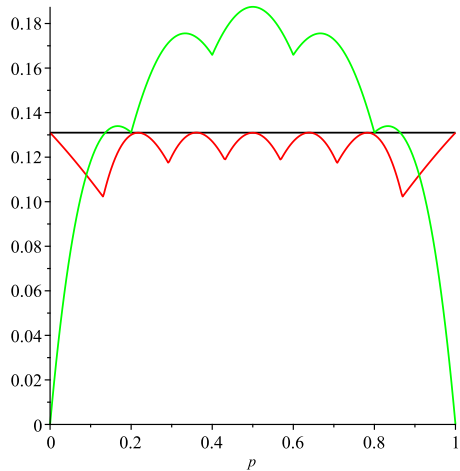}
\caption{Graph of $D(a_0,...,a_n;p)$ in the case of $n=5$ tosses: Maximum Likelihood (green), optimal choice (red), and the maximum value of the penalty function for the optimal choice (black) }
\label{FiveToss}
\end{figure}

It is worth pointing out that in the Statistics literature one often prefers the use of squares instead of absolute values in \eqref{penalty}, that is, the minimization of
\begin{equation}\label{quadraticloss}
\widehat{D}(a_0,...,a_n;p):=\sum_{k=0}^n {n \choose k} p^k(1-p)^{n-k} (p-a_k)^2\,,\,p\in [0,1]\,,
\end{equation}
is considered, because of its analytical tractability and easier computations. Indeed, for the penalty function \eqref{quadraticloss} the optimal (minimax) strategy $\{a_0,\ldots,a_n\}$ is explicitly computed (see \cite{LC}):
\begin{equation}\label{quadraticoptimal}
a_k = \,\frac{1}{2}\,+\,\frac{\sqrt{n}}{1+\,\sqrt{n}}\,\left(\frac{k}{n}\,-\frac{1}{2}\right)\,,\,k=0,\ldots,n\,.
\end{equation}
Of course, this is the optimal strategy when measuring the loss using the least squares norm, but not in our ``uniform'' setting.
In Figure 2 below, we augment Figure 1 with the plot of the penalty function $D$ for the strategy \eqref{quadraticoptimal}.

\begin{figure}[h]
\centering\includegraphics*[width=3.0in]{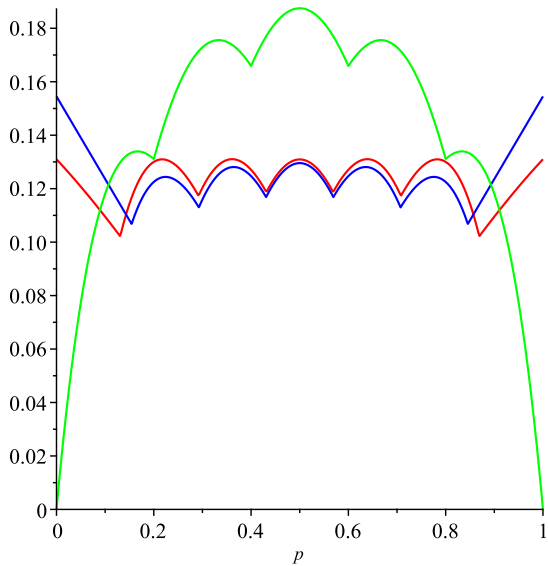}
\caption{Graph of $D(a_0,...,a_n;p)$ in the case of $n=5$ tosses: Maximum Likelihood (green), optimal choice (red), and the quadratic optimal choice (blue) }
\label{comparison}
\end{figure}

As it can be seen in Figure 2, the behavior of the Squared Error Minimax Estimator (hereafter, SEME) is similar, or even a bit better, towards the center of the interval, but clearly worse when we are close to the endpoints of the interval. Thus, for $n=5$, the sup-norm of the penalty function $D$ for the SEME is $0.1545$, while for our Absolute Error Minimax Estimator (AEME) is $0.131$. More generally, as mentioned above, along with the Minimax Estimators, the so-called Bayes Estimators are also often employed (see \cite[Ch. 4]{LC}). In this setting, given a loss function $R(\theta, \delta)$, some ``prior'' distribution $\Lambda$ for the parameter $\theta$ to be determined (in our case, $\theta = p$) is selected, and the estimator $\delta_{\Lambda}$ is chosen in order to minimize the weighted average risk
\begin{equation}\label{Bayesrisk}
r(\lambda, \delta) = \int \, R(\theta, \delta)\,d \Lambda (\theta),
\end{equation}
\noindent also called Bayes risk. As established in \cite[Theorem 5.1.4]{LC} and some corollaries, in certain situations a Bayes estimator is also a Minimax estimator. This connection will be used below (Section 4) when discussing the interpretation of our method within the framework of Game Theory. Thus, roughly speaking, we can say that in this paper we are dealing with a ``multi-faceted'' topic, which we investigate from the points of view of Point Estimation Theory, Approximation Theory, and Game Theory.

On the other hand, the problem of estimating the probability of a coin (more generally, the parameter of a Bernoulli distribution) from a few tosses has been often considered as a toy-model for randomization processes and for checking the effectiveness of different methods in statistical inference. In this sense, it is noteworthy that some recent papers have dealt with the problem of simulating a coin with a certain prescribed probability $f(p)$, by using a coin whose probability is actually $p$ (see \cite{Holtz} and \cite{Nacu}). Actually, this idea comes from a seminal paper by von Neumann \cite{Neumann}. In \cite{Holtz} and \cite{Nacu} the authors also show, just as we do in the current paper, the utility of techniques from Approximation Theory to solving problems from Probability.

The paper is structured as follows. In Section 2, our minimax estimation is thoroughly studied and the optimal choice is established by Theorem \ref{thm:equimax}, which represents the main result of this paper. Some computational results are included. The asymptotic distribution of the set of nodes corresponding to such optimal strategies, when the number of tosses approaches infinity, is established in Section 3. In Section 4 we discuss the problem from the Game Theory standpoint, as a non-cooperative two-player game, which is described in detail. Also, the existence of a Nash-equilibrium is established. Furthermore, in Section 5, this Nash-equilibrium is explicitly solved for the case of $n=2$ tosses.
In both Sections 4 and 5, the solution of this Nash-equilibrium problem is related to the connection between Minimax and Bayes estimators.
Finally, the last section contains some further remarks and conclusions.

\setcounter{equation}{0}
\setcounter{theorem}{0}
\section{The minimax estimation}
\label{Main}

\vskip 3mm

As it was discussed above, the optimal strategy we consider is to choose $a_0,\ldots,a_n \in [0,1]$ in order to minimize the sup-norm of the penalty function $D(a_0,\ldots,a_n;p)$. Thus, our minimax problem has a striking resemblance with the well-known problem of the optimization of the Lebesgue function in polynomial interpolation. Indeed, let us consider the polynomial interpolation of a function $f$ over a set of nodes $x_0,\ldots,x_n \in [0,1]$. By the classical Lagrange formula, we know the expression of such interpolating polynomial is: $$L_n (f;x_0,\ldots,x_n;x) = \sum_{k=0}^n\,f(x_k)\,l_k(x_0,\ldots,x_n;x)\,,$$ where $l_k(x_0,\ldots,x_n;x),\,k=0,\ldots,n,\,$ are the well-known Lagrange interpolation polynomials, and they form a basis for $\mathcal{P}_n$, the space of polynomials of degree less than or equal to $n$. Since the norm of the projection operator from $C[0,1]$, the space of all continuous functions on $[0,1]$, onto $\mathcal{P}_n$ is given by the sup-norm of the Lebesgue function

\begin{equation}\label{Lebesgue}
\Lambda(x_0,\ldots,x_n;x) = \sum_{k=0}^n\,|l_k(x_0,\ldots,x_n;x)|\,,
\end{equation}

\noindent the problem of finding optimal choices of nodes $x_0,\ldots,x_n \in [0,1]$ minimizing the sup-norm of \eqref{Lebesgue} arises in a natural way. It is well known that if the endpoints of the interval belong to the set of nodes, then the solution is unique. As for the characterization of the solution, the famous Bernstein-Erd\H{o}s conjecture asserted that for an optimal choice, the corresponding Lebesgue function \eqref{Lebesgue} must exhibit the following ``equimax'' property: If the absolute maximum of $\Lambda(x_0,\ldots,x_n;x)$ on each subinterval $[x_{i-1},x_i]$ is denoted by $\lambda_i = \lambda_i(x_0,\ldots,x_n)\,,i=1,\ldots,n$\, then we have:
\begin{equation*}\label{equimaxLeb}
\lambda_1 = \ldots = \lambda_n.
\end{equation*}
This conjecture was finally proved by Kilgore \cite{K} (see also \cite{BP}).

Now, the following result shows that the above mentioned resemblance between both minimax problems can be extended to the characterization of the optimal solutions. Let $f(p):=D(a_0,\ldots,a_n;p)$ be an optimal penalty function in the sense of minimizing the sup-norm of \eqref{penalty}. Then, the following result, which gathers some necessary conditions to be satisfied for an optimal choice $\{a_0,\ldots,a_n\}\,,$ will be useful. It will be stated without assuming that the points are ``well-ordered'', i.e. that $a_0 <a_1 <\ldots <a_n$. Although this fact may seem obvious, it does require a proof (see Theorem \ref{thm:equimax} below).

\begin{lemma}\label{lem:maxima}
Let
\begin{equation}M(f):=\{ x\in [0,1] \ : \ f(x)=\| f \|_\infty\}\end{equation}
be the set of absolute maxima of an optimal penalty function $f$. Then,
\begin{itemize}
\item [(i)] $M(f) \cap \{a_0,...,a_n\}\,=\,\emptyset$.
\item [(ii)] $M(f) \cap [0,\min \{a_i\})\,\neq\,\emptyset\,,\;M(f) \cap (\max \{a_i\},1]\,\neq\,\emptyset\,.$
\item [(iii)] $a_0 \leq 1/2 \leq a_n$ and $M(f) \cap (a_0, a_n)\,\neq\,\emptyset\,.$
\end{itemize}
\end{lemma}
\begin{proof}
The proof of part (i) easily follows from the fact that, from \eqref{penalty}, the derivative $f^\prime (p)$ has a positive jump as $p$ passes through $a_i$ and, thus, $f(p)$ cannot be increasing/decreasing as we pass through $a_i$.

\noindent As for part (ii), suppose that $M(f)\cap [0,\min \{a_i\}) = \emptyset$. Then, $\displaystyle \max_{[0,\min \{a_i\})} f(p)<\|f \|_\infty$. Since by (i), $\min \{a_i\} \not\in M(f)$, we have that there is a $\delta>0$ s.t. $\displaystyle \max_{[0,\min \{a_i\}+\delta]}f(p)<\|f \|_\infty$. But then, $\| D(a_0,a_1,\ldots,\min \{a_i\}+\delta,\ldots, a_n;p)\|_\infty<\|f\|_\infty$, which contradicts the optimality of $f(p)$. The argument that $M(f)\cap [\max \{a_i\},1)\not= \emptyset$ is similar.

\noindent To prove (iii) we first notice that $a_0\leq1/2\leq a_n$. Indeed, it is easily seen that \linebreak{$\|D(1/2,\dots,1/2;p)\|_\infty \leq 1/2$}, which implies that $\|f\|_\infty \leq 1/2$. Therefore, {$a_0=f(0)\leq \| f\|_\infty\leq 1/2$}. Similarly, $1-a_n = f(1)\leq 1/2$, and thus $a_0\leq 1/2\leq a_n$. We note that, as a consequence, we also have that $\min\{ a_i \}\leq a_0\leq 1/2\leq a_n\leq \max\{ a_i \}$.

\noindent Suppose now that $M(f)\cap[a_0, a_n]=\emptyset$.  Since $a_i\not\in M(f)$, there is $\delta>0$, such that $[a_0-\delta,a_n+\delta]\cap M(f)=\emptyset$. Using a similar argument as in the proof of (ii), we have that for $\epsilon>0$ small enough and $0\leq p \leq a_0 - \delta<1/2$,
\[
\begin{split}
D(a_0-\epsilon,\dots,a_n + \epsilon,\dots,a_n ; p)&=D(a_0,\dots,a_n;p)-\epsilon\left[ (1-p)^n-p^n \right] \\ &<D(a_0,\dots,a_n;p)
\end{split}\]

\noindent We can also prove the same inequality for $a_n+\delta < p\leq 1$ from which we conclude that

 \[
 \|D(a_0-\epsilon,a_1,\dots, a_{n-1}, a_n + \epsilon ; p)\|_\infty<\| D(a_0,\dots,a_n;p)\|_\infty=\|f \|_\infty,
  \]
\noindent which is a contradiction with the optimality of $f$.
\end{proof}

The following theorem establishes the {\em equimax property} for our minimax estimation and represents one of the two main results of this paper,
the other being Theorem \ref{LimDistr}. In addition, the ``well-ordering" of such optimal choice is also proved.

\vskip 3mm

\begin{theorem}\label{thm:equimax}
Suppose that
\begin{equation}f(p):=D(a_0, \dots , a_n;p)=D(T;p)\end{equation} is an optimal penalty function. Then the node set $T$ satisfies $a_0<a_1<\cdots <a_n $ and the {\em equimax property} holds, that is: $M(f)\cap (a_i ,a_{i+1} )\not= \emptyset$, $i=0,\dots,n-1$.
\end{theorem}

\begin{proof} First, we will prove that for an optimal penalty function $D(T;p)$ the node set $T$ is well-ordered, i.e. $a_i\leq a_{i+1}$ for all $i=0,\dots, n-1$. Indeed, suppose it is not. Then there is an index $i<n$ such that $a_i>a_{i+1}$. We will perturb the node set $T$ to obtain a penalty function with smaller norm.

\noindent Select $\epsilon>0$ small enough so that  $\|f \|>f(p)$ for $p\in(a_i -\psi, a_i +\psi) \cup (a_{i+1}-\delta, a_{i+1}+\delta)$, where

\begin{equation}\label{psidelta}
\psi:=\epsilon\,(a_i+a_{i+1})/(2-a_i -a_{i+1})\,,\;\delta:=\epsilon \, (i+1)/(n-i)
\end{equation}
and $\max(\psi,\delta)<(a_i-a_{i+1})/2$. Denote by $T_{\epsilon}$ the node set obtained by perturbing the nodes $a_i$ and $a_{i+1}$ to $a_i-\psi$ and $a_{i+1}+\delta$, respectively. Then, from \eqref{penalty}, we have

\begin{equation*}
D(T_\epsilon;p)-D(T;p) =
\epsilon\, {n \choose i}p^i (1-p)^{n-i} g(p),
\end{equation*}
where

\begin{equation*}
g(p):=\left \{ \begin{array}{lc}
\displaystyle{- \frac{a_i+a_{i+1}}{2-a_i-a_{i+1}}+\frac{p}{1-p}},\quad&p\leq a_{i+1}\\
&\\
\displaystyle{ -\frac{a_i+a_{i+1}}{2-a_i-a_{i+1}}-\frac{p}{1-p}},\quad&a_{i+1}+\delta \leq p \leq a_{i}-\psi\\
&\\
\displaystyle{\frac{a_i+a_{i+1}}{2-a_i-a_{i+1}}-\frac{p}{1-p}},\quad&p\geq a_{i}\,,\\
\end{array}\right.
\end{equation*}

\noindent and using the fact that $x/(1-x)$ is an increasing function, it is easy to see that $g(p)<0$ for all $p\in [0,1]\setminus \{(a_i -\psi, a_i) \cup (a_{i+1}, a_{i+1}+\delta) \}$. Additionally, as $f(p)<\|f\|$ on $(a_i -\psi, a_i +\psi) \cup (a_{i+1}-\delta, a_{i+1}+\delta)$, by selecting $\epsilon>0$ smaller if needed, we can guarantee $\| D(T_\epsilon;p)\|<\| f \|$, a contradiction with the optimality of $f$. This implies that for optimal penalty function the node set $T$ is well-ordered, i.e. $a_0<a_1<\cdots <a_n $.

\noindent Next, we prove the {\em equimax property}. Denote the global maxima on the consecutive subintervals by $\mu_{-1}:= \max_{[0,a_{0}]} f(p)$, $\mu_j:=\max_{[a_j,a_{j+1}]} f(p)$, $j=0,2,\dots, n-1$, and  $\mu_n:= \max_{[a_{n},1]} f(p)$. We want to show that
$\mu_0=\dots=\mu_{n-1}= \|f\|_{\infty}$ (we already know from Lemma \ref{lem:maxima} that $\mu_{-1}=\mu_{n}= \|f\|_{\infty}$).

\noindent By contradiction, assume that for some $i\in \{0,\ldots,n-1\}$ we have $\mu_i< \|f\|_{\infty}$. Then, as in the first part of the proof, we will construct a perturbation of the initial set of nodes, $T_{\epsilon}$, for which the corresponding penalty function has a smaller norm.

\noindent Fix $\epsilon>0$ small enough so that $\mu_i (T_\epsilon)<\|f\|_{\infty}$, where $T_\epsilon := \{a_0,\dots,a_i -\psi, a_{i+1} + \delta ,\dots, a_n \}$, with $\delta$ and $\psi$ given in \eqref{psidelta} (notice that now we are enlarging the original interval, while above it was shortened).
We will show that $\mu_k (T_\epsilon)<\mu_k (T)$ for all $k$, from which we will obtain $\|D(T_\epsilon;p \|< \|f\|_{\infty}$, a contradiction with the optimality of $f$.

\noindent Indeed, let $q \in [a_k,a_{k+1}]$ be such that $D(T_\epsilon;q)=\mu_k (T_\epsilon)$. Then

$$D(T_\epsilon;q)-D(T;q) =
\epsilon\, {n \choose i}q^i (1-q)^{n-i} h(q),$$ where

\begin{equation*}
h(q):=\left \{ \begin{array}{lc}
\displaystyle{- \frac{a_i+a_{i+1}}{2-a_i-a_{i+1}}+\frac{q}{1-q}},\quad&q<a_i\\
&\\
\displaystyle{\frac{a_i+a_{i+1}}{2-a_i-a_{i+1}}-\frac{q}{1-q}},\quad&q>a_{i+1}\,,\\
\end{array}\right.
\end{equation*}

\noindent and, thus, the fact that $x/(1-x)$ is an increasing function implies that $\mu_k (T_\epsilon)<\mu_k\leq \|f\|_{\infty}$ for all $k\not= i$. Since the choice of $\epsilon$ implies that $\mu_i (T_\epsilon)< \|f\|_{\infty}$, then we obtain the desired contradiction.
\end{proof}

\begin{example} {\bf The Optimal Penalty Function for some values of $n$.} {\rm For illustration, in Figures \ref{3tossplot} - \ref{7tossplot} we present the computational results we obtained for the cases of $n=3, 4, 6$, and $7$ tosses (the plot for $n=5$ was already shown in Figure \ref{FiveToss}). They all confirm the conclusions of Theorem \ref{thm:equimax}. In all of these figures, we plot the optimal choice for the penalty function $D(a_0,...,a_n;p)$, that is, AEME, (red), the Maximum Likelihood function (green), and the maximum value of the optimal penalty function (black).}
\end{example}


\begin{figure}
\centering
\begin{minipage}{.5\textwidth}
  \centering
  \includegraphics[width=.7\linewidth]{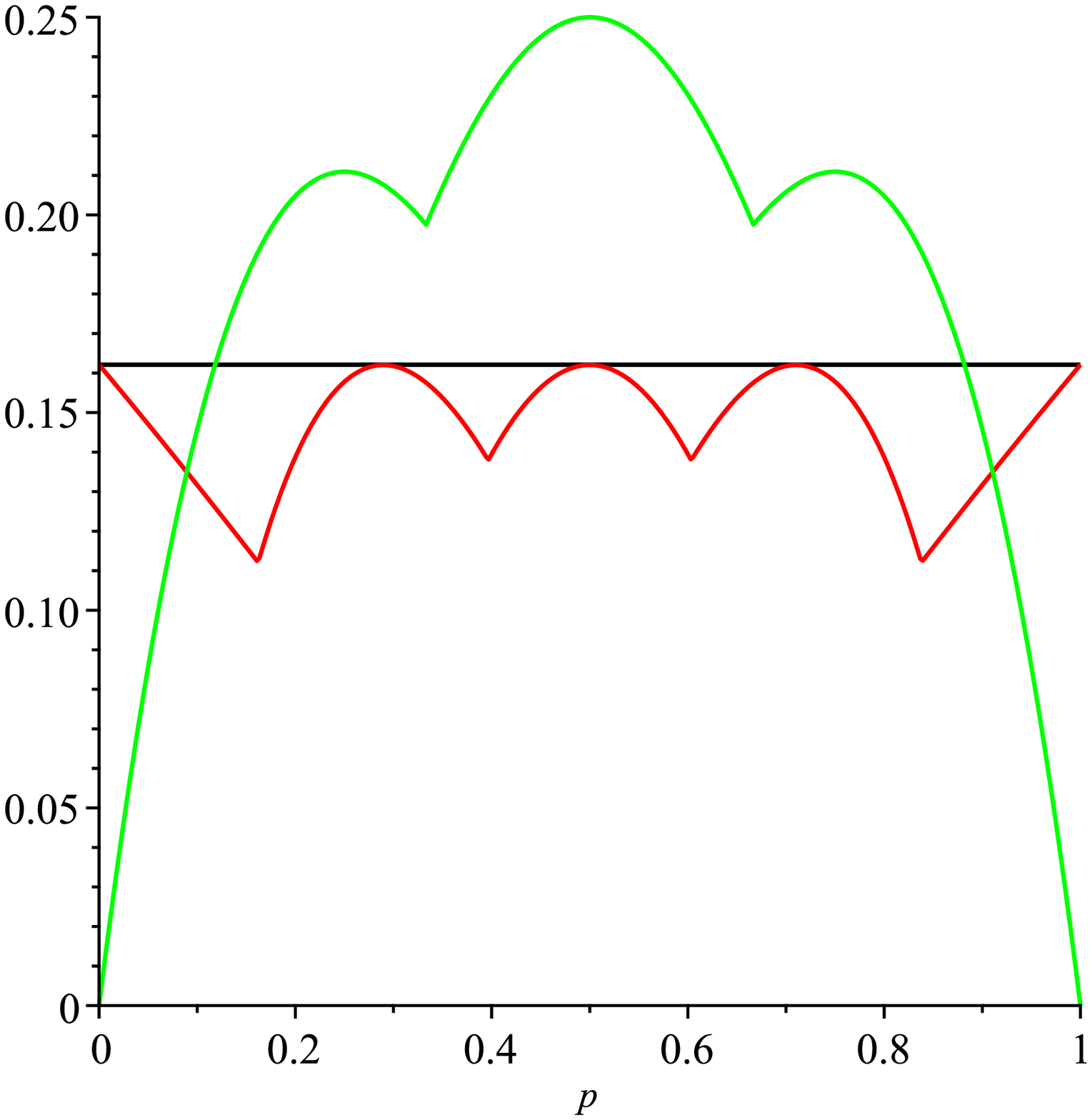}
  \caption{The case of $n=3$ tosses}
  \label{3tossplot}
\end{minipage}%
\begin{minipage}{.5\textwidth}
  \centering
  \includegraphics[width=.7\linewidth]{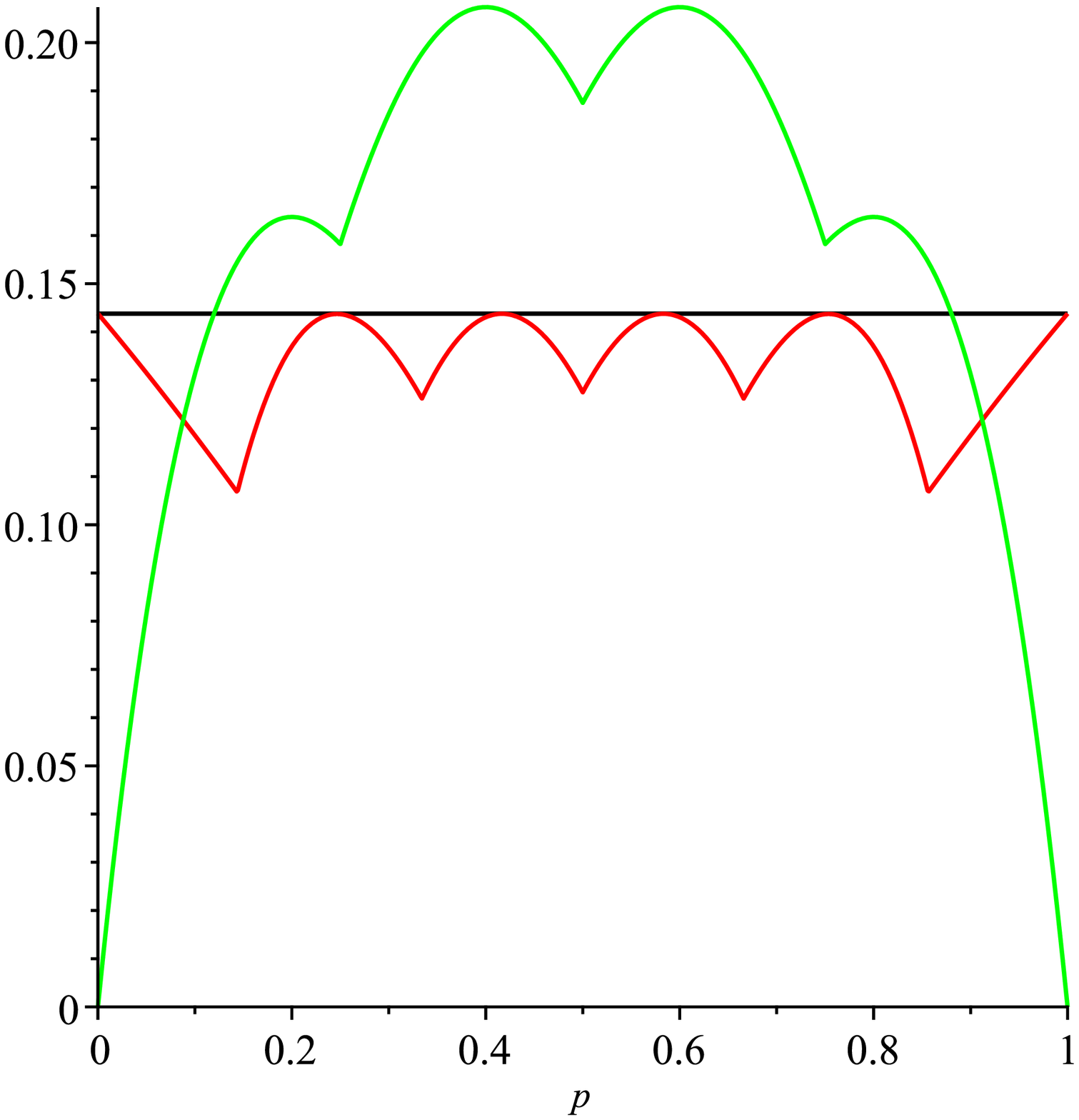}
  \caption{The case of $n=4$ tosses}
  \label{4tossplot}
\end{minipage}
\end{figure}

\begin{figure}
\centering
\begin{minipage}{.5\textwidth}
  \centering
  \includegraphics[width=.7\linewidth]{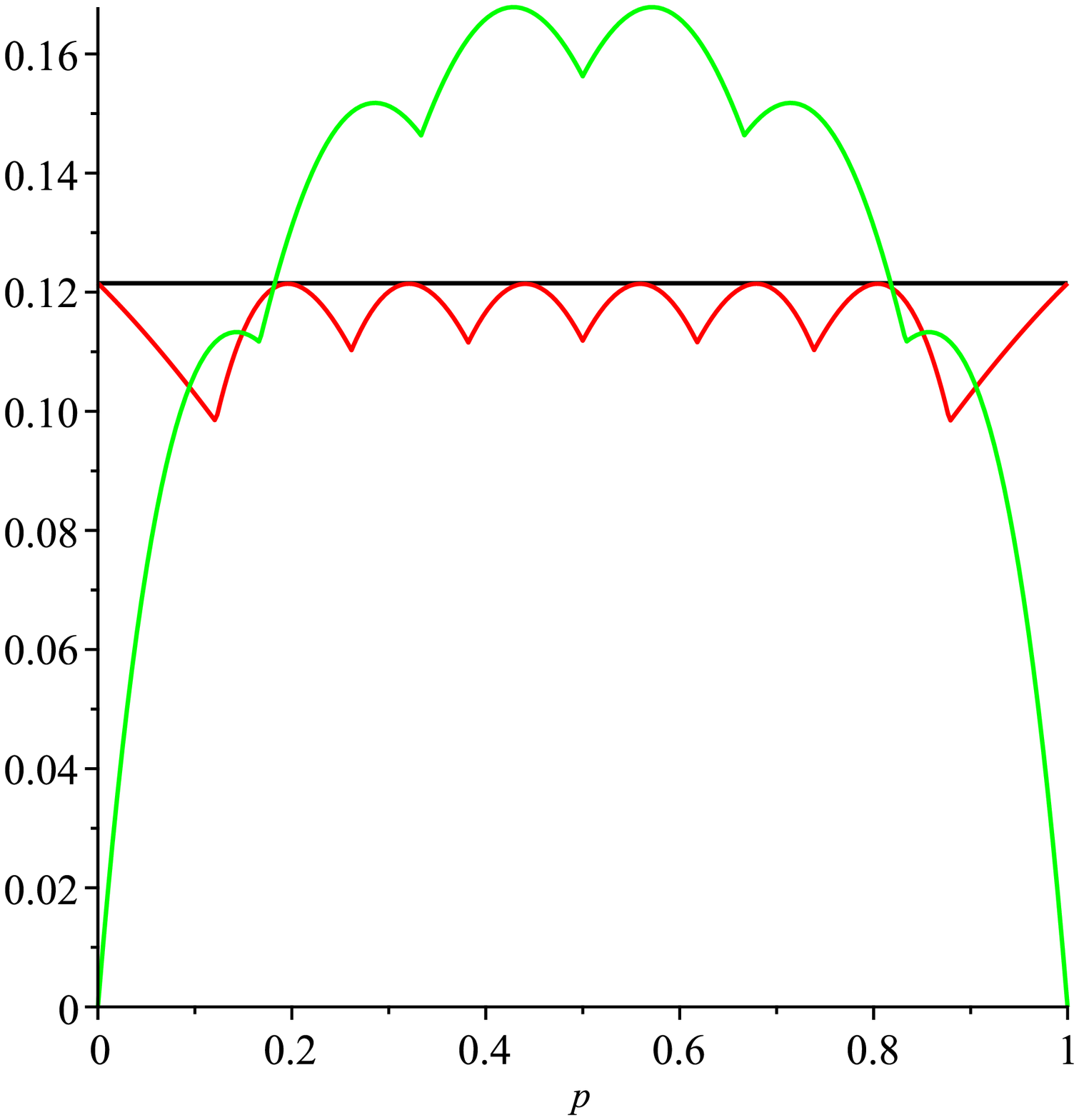}
  \caption{The case of $n=6$ tosses}
  \label{6tossplot}
\end{minipage}%
\begin{minipage}{.5\textwidth}
  \centering
  \includegraphics[width=.7\linewidth]{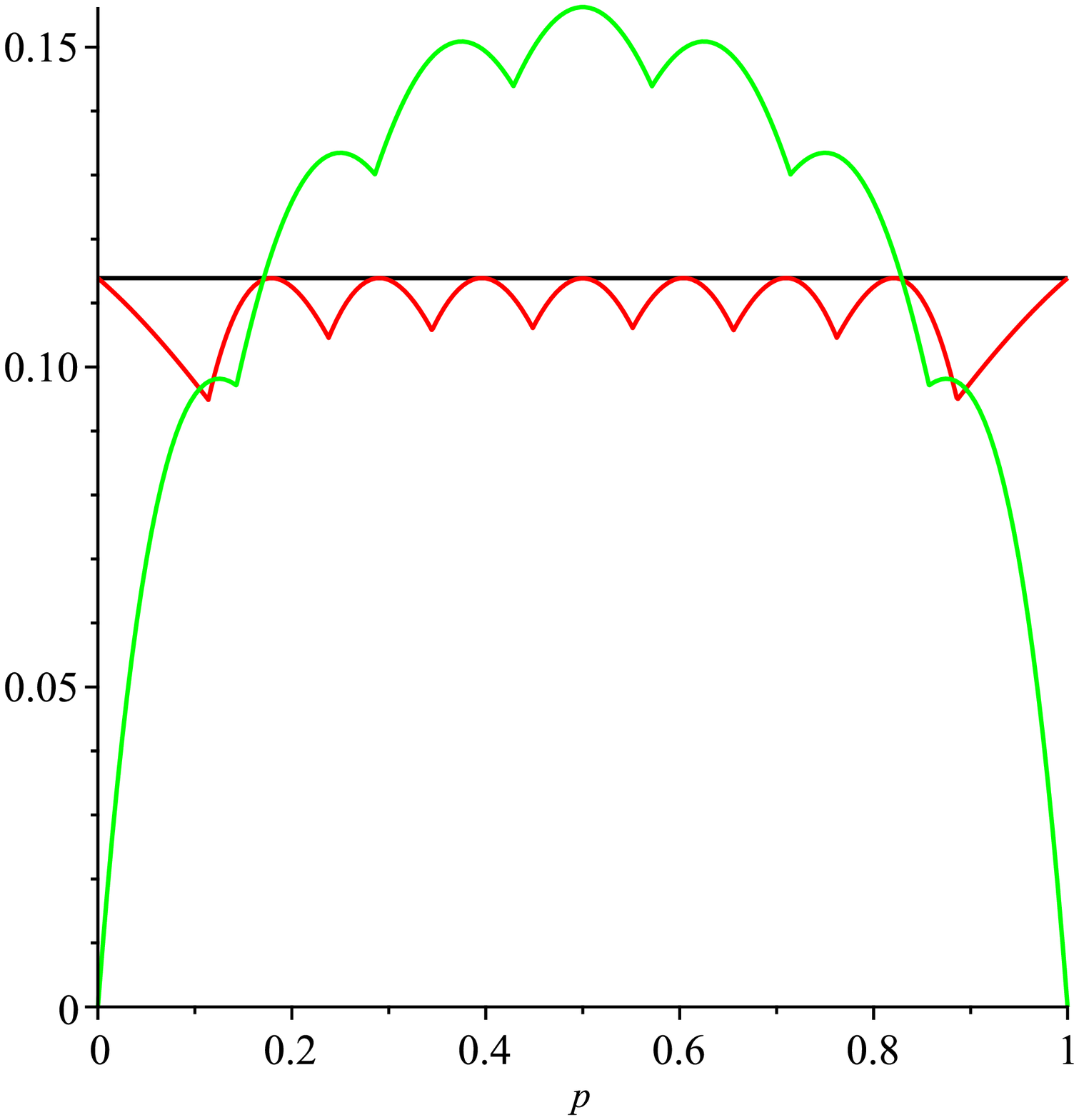}
  \caption{The case of $n=7$ tosses}
  \label{7tossplot}
\end{minipage}
\end{figure}

\begin{remark}\label{constantrisk}
It is relevant to point out here that when using the penalty function $\widehat{D}$ corresponding to the squared error loss function (see \eqref{quadraticloss}), it is shown in \cite[Ch. 5]{LC} that the related minimax estimation (SEME), given by \eqref{quadraticoptimal}, has a constant risk function, as shown in Figure \ref{fig:Dhat}, where the risk function for SEME (blue) is compared with those for AEME (red) and for the Maximum Likelihood (hereafter, MLE).

Since a constant function obviously satisfies the equimax property, this supports our conjecture that this equimax characterization should hold for Minimax Estimators corresponding to any convex loss function.

\begin{figure}[h]
\centering\includegraphics*[width=2.0in]{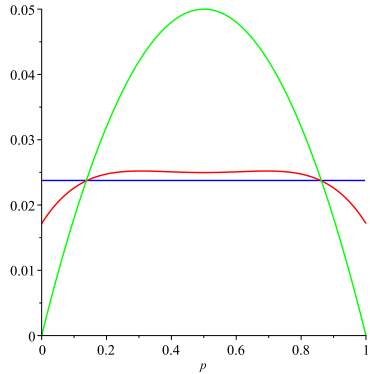}
\caption{Graph of $\widehat{D}(a_0,...,a_n;p)$ in the case of $n=5$ tosses: Maximum Likelihood (green), Absolute Error Minimax Estimator (red), and the Squared Error Minimax Estimator (blue) }
\label{fig:Dhat}
\end{figure}

\end{remark}

\vspace{.5cm}

\setcounter{equation}{0}
\setcounter{theorem}{0}
\section{The asymptotic behavior of the minimax estimations}
\label{Limit}

In this section we shall consider the limiting distribution of the minimax estimations as the number of tosses $n$ approaches infinity.

\begin{definition} For every positive integer $n$ let $A_n:=\{ a_{0n}, a_{1n}, \dots, a_{nn} \}$ denote a {\em minimax estimations set} or, simply, a {\em minimax node set}, namely

\[\|D(A_n,p)\|:=\|D(a_{0n}, \dots, a_{nn};p) \|=\,\min_{x_0,...,x_n \in [0,1]} \| D(x_0,...,x_n;\cdot)\|_\infty .\]

\end{definition}

Observe, that by Theorem \ref{thm:equimax} the node set $A_n$ is ordered. Moreover, if $B_n:=\{k/n\}_{k=0}^n$ denotes the {\em uniform node set} corresponding to the Maximum Likelihood estimation, then
\begin{equation}\label{Cond1}\|D(A_n, p)\|\leq \|D(B_n, p)\|=\mathcal{O} \left(\frac{1}{\sqrt{n}} \right)\longrightarrow 0\quad {\rm as}\quad n\to  \infty.\end{equation}

Indeed, the $\mathcal{O}(1/\sqrt{n})$ estimate follows as an application of the Jensen's inequality to the convex function $f(x)=x^2$
\begin{equation*}
\begin{split}
D(B_n,p)^2&= \left( \sum_{k=0}^n {n \choose k} p^k(1-p)^{n-k} \left| p-\frac{k}{n}\right| \right)^2 \\
&\leq  \sum_{k=0}^n {n \choose k} p^k(1-p)^{n-k} \left( p-\frac{k}{n} \right)^2 = \frac{p(1-p)}{n}\leq \frac{1}{4n},
\end{split}
\end{equation*}
where we used the fact that the mean of the binomial distribution is $\mu=np$ and its variance is $\sigma^2=np(1-p)$.

\vspace{.25cm}

We will prove (see Theorem \ref{LimDistr} below) that the ordering of $A_n$, established in Theorem 3.1, and equation \eqref{Cond1} yield that the limiting distribution is uniform. For this purpose let us remind the reader the definition of weak$^*$ convergence of a sequence of measures.

\begin{definition} \label{Weak*Conv} Let $\{\mu_n \}$ be a sequence of measures supported on $[0,1]$. We say that it converges weakly (or weak$^*$) to a measure $\mu$ if
\begin{equation}\label{weak*1}
\lim_{n\to\infty} \int_0^1 f(t)\, d \mu_n (t) =\int_0^1 f(t)\, d \mu (t) \quad {\rm for\ \ all}\quad f\in C[0,1],
\end{equation}
where $C[0,1]$ denotes all continuous functions on $[0,1]$, or equivalently
\begin{equation} \label{weak*2}
\lim_{n\to \infty} \mu_n ([a,b]) = \mu([a,b]) \quad {\rm for\ all} \quad [a,b]\subset [0,1].
\end{equation}
We denote this as
\[ \mu_n \stackrel{*}{\longrightarrow} \mu, \quad {\rm as}\quad  n\to \infty.\]
\end{definition}

\begin{definition} \label{CountingMeasDef}
Given a finite set $K_n := \{ \alpha_{0n},\alpha_{1n}, \dots, \alpha_{nn} \}$ we call the measure
\begin{equation}
\delta_{K_n} := \frac{1}{n+1}\sum_{j=0}^n \delta_{\alpha_{jn}},
\end{equation}
a  {\em normalized counting measure} of $K_n$. Here $\delta_x$ denotes the Dirac-delta measure at the point $x$.
\end{definition}

\begin{theorem}\label{LimDistr}
Suppose that the node sets $K_n := \{ \alpha_{0n},\alpha_{1n}, \dots, \alpha_{nn} \}\,\subset\,[0,1]$, $n=1,\dots,\infty$, are ordered and that
\begin{equation}\label{Cond2}\lim_{n\to \infty} \|D(K_n, p)\|= 0.\end{equation}

Then the asymptotic distribution of $K_n$, as $n$ tends to infinity, is uniform, namely
\begin{equation}\label{ConvThm} \delta_{K_n} \stackrel{*}{\longrightarrow} dx, \quad {\rm as}\quad  n\to \infty,\end{equation}
where $dx$ denotes the Lebesgue measure on the interval $[0,1]$.
\end{theorem}

\begin{proof}
To prove \eqref{ConvThm} it is sufficient to establish that for all $0<p<1$ we have
\[p\leq \liminf_{n\to \infty} \frac{|K_n \cap [0,p]|}{n+1} \leq \limsup_{n\to \infty} \frac{| K_n \cap [0,p] |}{n+1} \leq p .\]
We shall prove the third inequality, the first being similar and the second being obvious.

Suppose that there is a $0<p<1$  for which it fails. Then there is an $\epsilon >0$ such that
\[\limsup_{n\to \infty} \frac{| K_n \cap [0,p] |}{n+1} > p +\epsilon.\]
This implies that there is a subsequence $\{n_i\}$ and a number $M$, such that
\begin{equation}\label{inequality}
\frac{| K_{n_i} \cap [0,p] |}{n_i +1} > p +\frac{\epsilon}{2}=:q \quad {\rm for \ all} \quad i \geq M .
\end{equation}
For every $i$ denote $k_i :=| K_{n_i} \cap [0,p] |$. Then the ordering of $K_{n_i}$ implies that $|q-\alpha_{\ell,n_i}|>\epsilon/2$ for all $\ell\leq k_i$ and hence
\begin{equation}
\begin{split}
D(K_{n_i},q)&=\sum_{\ell=0}^{n_i} {n_i \choose \ell} q^\ell(1-q)^{n_i-\ell} |q-\alpha_{\ell,n_i}| \\
&\geq \sum_{\ell=0}^{k_i} {n_i \choose \ell} q^\ell(1-q)^{n_i-\ell} |q-\alpha_{\ell,n_i}| \\
& > \frac{\epsilon}{2}\sum_{\ell=0}^{k_i} {n_i \choose \ell} q^\ell(1-q)^{n_i-\ell}>\frac{\epsilon}{4},
\end{split}
\end{equation}
where in the last inequality we apply \cite[Theorem 1]{KB}, using the fact that \eqref{inequality} yields that $k_i-qn_i>q$. This contradicts \eqref{Cond2}, which proves the theorem.
\end{proof}

As the minimax node sets $A_n$ are ordered, \eqref{Cond1} allows us to establish the following.
\begin{corollary} \label{cor:weakconv}
The limiting distribution of the minimax node sets $A_n$ is the uniform distribution $dx$ on $[0,1]$.
\end{corollary}

\begin{remark}\label{rem:universality}

Observe that Theorem \ref{LimDistr} establishes that the limiting distribution is the uniform one not just for the sequence of optimal choices, but for every sequence of ``acceptable'' strategies, in the sense that they are well ordered and the sup-norm of their corresponding penalty functions approaches zero, as $n\rightarrow \infty$. Thus, we see that all these acceptable estimators are asymptotically unbiased, in the sense that their limit distribution as $n\rightarrow \infty$, is the uniform distribution, the same as for the Maximum Likelihood estimators, which are the unique unbiased estimators for the parameter $p$.

It is also remarkable that this conclusion also holds for the sequence of Squared Error Minimax Estimators given by \eqref{quadraticoptimal}, which is easy to check.

\end{remark}

\setcounter{equation}{0}
\setcounter{theorem}{0}
\section{A problem in Game Theory}
\label{Description}

As it was said above, now we are going to see our estimation problem from the viewpoint of the Game Theory (see e.g. \cite{BG}, \cite{Fudenberg} or \cite{Osborne}). In particular, a non-cooperative, two-player, zero-sum and mixed-strategy game will be posed, as we explain below.

Indeed, we are dealing with a simple two-player game, where Player I selects a probability $p\in [0,1]$ and creates a
coin such that $P(heads)=p$. He tosses the coin $n$ times and provides the number $i\in \{0,1,...,n\}$ of heads observed to Player II. Then, based on this value,
Player II makes a guess $a_i\in [0,1]$ for the value of $p$ and he will pay a loss of $|p-a_i|$  to Player I. Obviously, the goal of Player I is to maximize this loss, while
Player II wants to minimize it.

More generally, let us assume that both players are allowed to follow what is commonly referred to as a ``mixed strategy'' in the Game Theory framework: that is, the choices of the players are not deterministic (``pure strategy''), but the available actions are selected according to certain probability distributions. Thus, let $\Omega$ denote the set of all probability distributions on the interval $[0,1]$. Suppose that when
making their decisions, Player I is allowed to choose $\mu \in \Omega$ and he picks $p$ to
follow the distribution $d\mu$, and Player II picks $x_i$ to follow his choice of $d\sigma_i$
distributions, where $\sigma_i \in \Omega, \ i=0,1,...,n$.

Therefore, the expected penalty of Player II is
\begin{equation}
E(\sigma_0,...,\sigma_n;\mu):=\int \cdots \int D(x_0,...,x_n;p) d\sigma_0(x_0) ... d\sigma_n(x_n) d\mu(p),
\end{equation}
where $D(x_0,...,x_n;p)$ is the penalty function given in \eqref{penalty}. In these terms, the goal of Player I will be to find
\begin{equation} \label{maximin} \max_{\mu\in\Omega} \min_{\sigma_0,...,\sigma_n \in\Omega} E(\sigma_0,...,\sigma_n;\mu)\,,\end{equation}
while the second player will try to get
\begin{equation} \label{minimax} \min_{\sigma_0,...,\sigma_n \in\Omega} \max_{\mu\in\Omega} E(\sigma_0,...,\sigma_n;\mu\,).\end{equation}

Using again the terminology from Game Theory, this is a ``zero-sum'' game (that is, the total gains of players minus the total losses add up to zero). Now, we are looking for the so-called mixed-strategy Nash-equilibrium. The basic notion of equilibrium in Game Theory finds its roots in the work by Cournot (1838), but was formalized in the celebrated papers by Nash, \cite{N1}-\cite{N2}, where the (now called) Nash-equilibrium was established for finite games by using Kakutani's Fixed Point Theorem \cite{Ka}. The problem for continuous games, where the players may choose their strategies in continuous sets (as in the present case), is more involved.

The following result, establishing the mixed-strategy Nash-equilibrium for our problem, is a direct application of the Glicksberg's Theorem \cite{Glicksberg}, who made use of an extension of Kakutani's Theorem to convex linear topological spaces. An alternate method of proof consists in taking finer and finer discrete approximations of our continuous game, for which the existence of a Nash-equilibrium was established, and then using the continuity of \eqref{penalty} and standard arguments of weak convergence (for more information about the successive extensions of the Nash-equilibrium problem, see for example the monographs \cite{Fudenberg} and \cite{Osborne}; there are also many papers about such extensions from the point of view of the applications to Business, see e.g. \cite{McKe} and \cite{COP}, to only cite a few).

\begin{theorem}
\begin{equation} \label{Nash generalized}
\min_{\sigma_0,...,\sigma_n \in\Omega} \max_{\mu\in\Omega} E(\sigma_0,...,\sigma_n;\mu) =
\max_{\mu\in\Omega} \min_{\sigma_0,...,\sigma_n \in\Omega} E(\sigma_0,...,\sigma_n;\mu).
\end{equation}
\end{theorem}

\vskip 3mm

For our analysis it is important to make use of the following discretization of the probability distribution setting in \eqref{Nash generalized} related to the second player's strategy.
\begin{theorem}\label{thm:discret} The minimax problem for distributions \eqref{Nash generalized} admits the following discretization
\begin{equation}\label{discret}
\begin{split}
\min_{\sigma_0,...,\sigma_n \in\Omega} \max_{\mu\in\Omega} E(\sigma_0,...,\sigma_n;\mu) & = \min_{a_0,...,a_n \in [0,1]} \max_{p\in [0,1]} D(a_0,...,a_n;p) \\
& \\
& =\min_{a_0,...,a_n \in [0,1]} \| D(a_0,...,a_n;\cdot)\|_\infty\,.
\end{split}
\end{equation}
\end{theorem}
\begin{proof}
Let $a_i:=\int_0^1 \theta \, d\sigma_i (\theta)$, $i=0,1,\dots,n$. Since $\displaystyle \int_0^1\,|p-\theta|\,d\sigma_i(\theta)\,\geq\,|p-a_i|\,,$ we have that
$$E(\sigma_0,...,\sigma_n;\mu)\,\geq\,E(a_0,...,a_n;\mu)\,.$$
Then, by the continuity of $D(a_0,\dots,a_n;p)$, we get
$$ \max_{\mu\in\Omega} E(a_0,...,a_n;\mu)= \| D(a_0,...,a_n;\cdot)\|_\infty.$$
Hence,
\begin{equation}\label{relation1}
\min_{\sigma_0,...,\sigma_n \in\Omega} \max_{\mu\in\Omega} E(\sigma_0,...,\sigma_n;\mu)\,\geq\,\min_{a_0,...,a_n} \max_{p\in [0,1]} D(a_0,...,a_n;p)
\end{equation}
On the other hand, if for fixed points $a_0,\ldots,a_n\in [0,1]$, we take $\displaystyle \sigma_i = \delta_{a_i}\,,$ it is clear that
$E(\sigma_0,...,\sigma_n;\mu) = E(a_0,...,a_n;\mu)$ and, thus
\begin{equation}\label{relation2}
\min_{\sigma_0,...,\sigma_n \in\Omega} \max_{\mu\in\Omega} E(\sigma_0,...,\sigma_n;\mu)\,=\,\min_{a_0,...,a_n \in [0,1]} \max_{\mu\in\Omega} E(a_0,...,a_n;\mu)\,,
\end{equation}
but, for some $\mu^* \in \Omega$,
\begin{equation}
\begin{split}\label{relation3}
\max_{\mu\in\Omega} E(a_0,...,a_n;\mu) &= \int_0^1\,D(a_0,...,a_n;p)\,d\mu^*(p)\\
&\leq\,\max_{p\in [0,1]} D(a_0,...,a_n;p)
=\,\|D(a_0,...,a_n;\cdot)\|_{\infty}\,,
\end{split}
\end{equation}
and this settles the proof of \eqref{discret}.
\end{proof}

\vskip 3mm

\begin{remark}\label{doneinsection2}
The above discretization \eqref{discret} shows that the optimal strategy for Player II described in the previous section (Theorem \ref{thm:equimax}) agrees with the set of the Absolute Error Minimax Estimators (AEME), using the language of Point Estimation Theory.

\end{remark}

Therefore, our main concern now is the Optimal Strategy for the Player I. But in this sense, the arguments used in the proof of Theorem \ref{thm:discret} also have an important consequence for the Player I's strategy. Indeed, if we denote by $\mu^*$ an optimal distribution for the first player and by $\supp \mu^* \subset [0,1]$, its support, then we have
\begin{lemma}\label{lem:support}
\begin{equation}\label{supp}
\supp \mu^* \subset M(f)
\end{equation}
\end{lemma}

\begin{proof}It is enough to realize that in the proof of Theorem \ref{thm:discret}, equations \eqref{relation1}-\eqref{relation2} show that for an extremal measure $\mu^*$ (for which the Nash-equilibrium \eqref{Nash generalized} is attained), \eqref{relation3} is actually an equality and, therefore, $\supp \mu^*$ must be contained in the set where the equality
$f(p)=\|f\|_\infty = \|D(a_0,\ldots,a_n;.)\|_{\infty}\,$ is attained (with $(a_0,\ldots,a_n)$ being an optimal choice for Player II). This establishes \eqref{supp}.
\end{proof}

Theorem \ref{thm:discret} and Lemma \ref{lem:support} show that, from this point on, we may assume that both players follow strategies based on discrete distributions. Indeed, while a Player II's optimal (pure) strategy will be a choice $\{a_0,\ldots,a_n\} \subset [0,1]\,,$ an optimal (mixed) strategy for the Player I will be based on an atomic measure $\displaystyle \mu^* = \sum_{j=1}^k\,m_j\,\delta_{p_j}\,,$ where $\displaystyle \{p_j\}_{j=1}^k \subset M(f)$ and $\sum_{j=1}^k\,m_j = 1\,.$

\begin{example}{The case of $n=1$ toss.{\rm }}
\label{1 toss}
\vskip 2mm
\noindent {\rm Because of the simplicity and the symmetry of the problem, the method to find the strategies satisfying the Nash-equilibrium (\ref{Nash generalized}) can be carried out easily in this simple case by using Lemmas \ref{lem:maxima} and \ref{lem:support} above. Therefore, we skip the details.

\noindent The optimal  discrete strategy $\mu$ of Player I is the following: choose the atomic measure $\displaystyle \mu = \sum_{k=0}^2\,m_i\,\delta_{p_i}$, with $p_0=0, p_1=0.5, p_2=1$ and
the corresponding weights given by $m_0=0.25, m_1=0.5, m_2=0.25$.
Further, for $a_0\in [0,0.5], a_1\in [0.5,1]$, we have $E(a_0,a_1;\mu)=0.25$ and for any $a_0,a_1\in [0,1]$ we have $0.25\le E(a_0,a_1;\mu)$.
Thus, for any $\sigma_0,\sigma_1\in $ we have $0.25\le E(\sigma_0,\sigma_1;\mu)$.

\noindent The optimal discrete strategy of Player II is the following: $\sigma_0$ is the unit point mass at $a_0=0.25$, and  $\sigma_1$ is the unit point mass at $a_1=0.75$.
The graph of $f(p)=D(0.25,0.75;p)=(1-p)|p-0.25|+p|p-0.75|$ is shown in Figure \ref{1tossplot}.
Since for all $p\in[0,1]$ we have $f(p)\le 0.25$, we conclude that for any $\mu\in\Omega$, $E(\sigma_0,\sigma_1;\mu)\le 0.25$.
So the Nash-equilibrium is established, and if both players follow the outlined strategies, then Player II pays $E(a_0,a_1;\mu)=0.25$ dollars to Player I.}

\begin{figure}[h]
\centering\includegraphics*[width=3.0in]{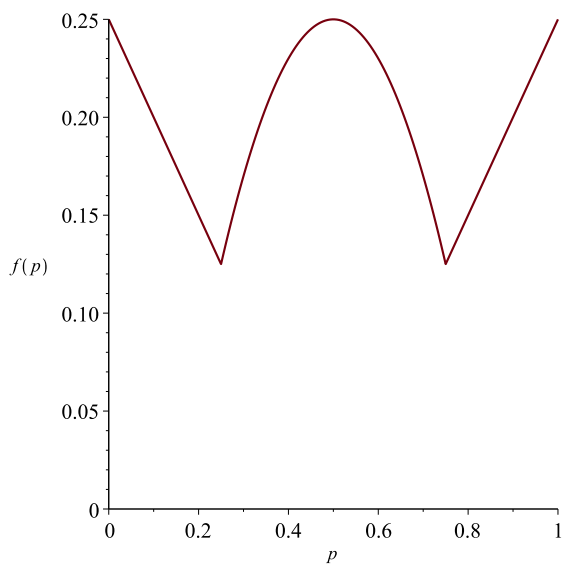}
\caption{Graph of $f(p)$ in the case of $n=1$ toss.}
\label{1tossplot}
\end{figure}
\end{example}\

\begin{remark}\label{rem:leastfav}
Going back to the Point Estimation Theory approach, the results in previous Theorem 4.2 and Lemma 4.4 admit an interpretation in terms of the connection between Bayes and Minimax estimators, as
mentioned in the Introduction. Indeed, Theorem 5.1.4 and especially Corollary 5.1.5 in \cite{LC} establish sufficient conditions to ensure that a Bayes estimator is also a Minimax one, namely, that the average risk (or penalty) of the Bayes estimator $\delta_{\Lambda}$ for a certain prior distribution $\Lambda$ (see \eqref{Bayesrisk}) agrees with the value of the maximum of that risk and, hence, that for such distribution $\Lambda$, the risk function must be constant. If this is the case, the prior distribution $\Lambda$ is said to be a \emph{least favorable} one. In this sense, our results are a sort of converse of the ones in \cite{LC}.

When the squared error loss function is used, one could see in Figure \ref{fig:Dhat} that the risk function for the corresponding Minimax estimator is constant throughout the interval $[0,1]$. In this case, it is proven in \cite[Ex. 5.1.7]{LC} that this Minimax estimator is indeed a Bayes estimator with respect to a continuous distribution supported in $[0,1]$, namely the Beta distribution $\displaystyle \mathcal{B} \left(\frac{\sqrt{n}}{2}, \frac{\sqrt{n}}{2}\right)$. Therefore, this last distribution plays the role of a least favorable one and its corresponding Bayes estimator is proven to be unique, hence, by \cite[Th. 5.1.4]{LC}, it is also unique as Minimax estimator.
Taking into account our previous results, this least favorable distribution would represent the optimal strategy for Player I in this case.

The above connection for our absolute error loss function is more involved and it will be discussed in the next section, where the Nash-equilibrium for the case of $n=2$ tosses is thoroughly analyzed.

\end{remark}

\setcounter{equation}{0}
\setcounter{theorem}{0}
\section{A constructive proof of the Nash-equilibrium for the case of $n=2$ tosses}
\label{Main1}

Now, we are concerned with the existence and uniqueness of a strategy pair solving the Nash equilibrium \eqref{Nash generalized} in the case of $n=2$ tosses. Our method will be based on previous Theorem \ref{thm:equimax} and Lemmas \ref{lem:maxima}-\ref{lem:support}.

Recall that the strategy of Player II is to minimize the maximum of the penalty function, namely determine optimal outcomes $\{a_0^*, a_1^*, a_2^* \}$ defining an optimal penalty function
$f(p):=D(a_0^*,a_1^*,a_2^*;p)$ such that

\begin{equation}
\label{PlayII}
\min_{\{a_0, a_1, a_2 \}\subset [0,1]} \max_p D(a_0, a_1, a_2;p)=\min_{\{a_0, a_1, a_2 \}\subset [0,1]} \| D(a_0,a_1,a_2;p) \|_\infty= \|f\|_\infty\,.
\end{equation}

\noindent That such an $f$ exists follows easily by a compactness argument.

The strategy of Player I is to find a probability measure $d\mu^*(p)$, supported on $[0,1]$, that maximizes the expected penalty no matter what the choice of Player II is, i.e. determine
\begin{equation}
\label{PlayI}
\mathcal{F}:=\max_{\mu} \min_{\{a_0, a_1, a_2 \} \subset [0,1]}\int_0^1 D(a_0, a_1, a_2; p) d\mu (p).
\end{equation}
Clearly, for any $\mu \in \Omega$,
\[\int_0^1 D(a_0,a_1,a_2;p) d\mu (p) \leq \|D(a_0,a_1,a_2;p)\|_\infty ,\]
so,
\[ \min_{\{a_0, a_1, a_2 \} \subset [0,1] }\int_0^1 D(a_0,a_1,a_2;p) d\mu (p)\leq \|f\|_\infty ,\]

\noindent which implies, after taking the $\max$ over all $\mu$, that $\mathcal{F}\leq \|f\|_\infty$. Then, our goal in this section is to find an optimal strategy pair
$\{a_0^*, a_1^*, a_2^* \}\subset [0,1],\,\mu^* \in \Omega,$ for which the Nash-equilibrium, $\mathcal{F} = \|f\|_\infty$, is uniquely reached.

\vskip 1mm

The main result in this section is stated as follows.

\begin{theorem}\label{thm:2tosses}
The Nash-equilibrium is reached if the players use the following strategies:
\begin{itemize}

\item [\textbf{Player I:}] Choose $p$ according to the distribution $\displaystyle \mu = \sum_{i=0}^3\,m_i\,\delta_{p_i}\,,$ where
\begin{equation}\label{p1}
p_1 = \frac{1}{3} \Big( 1+ \sqrt[3]{1+3\sqrt{57}} - \frac{8}{\sqrt[3]{1+3\sqrt{57}}} \Big) \approx 0.3611
\end{equation}
is the unique real root of the polynomial $x^3-x^2+3x-1$, and
\begin{equation}p_2 = 1-p_1 \approx 0.6389\,,\,p_0 = 0, p_3 = 1\,.
\end{equation}
The weights $m_i$ are given by:
\begin{equation}\label{weights}
\begin{split}
m_1&=\frac{0.5}{p_1^2+(1-p_1)^2+1} \approx 0.325 \\
m_0 &= 0.5 - m_1 \approx 0.175\,,\;m_2 = m_1\,,\;m_3 = m_0\,.
\end{split}
\end{equation}

\item [\textbf{Player II:}] Choose the following values $a_i$
\begin{equation}\label{ai}
a_0 = \frac{2p_1(1-p_1)^2}{p_1^2+(1-p_1)^2+1}  \approx 0.1916\,,\;a_2 = 1- a_0 \approx 0.8084\,,\; a_1 = 0.5\,.
\end{equation}

\end{itemize}

Furthermore, the above pair of strategies is unique in the following sense: if the choice of distributions $\{\sigma_0, \sigma_1, \sigma_2,\mu \}$ satisfies the Nash-equilibrium \eqref{Nash generalized}, then $E(\sigma_i)=a_i, i=0,1,2$, and $\displaystyle \mu= \sum_{i=0}^3 m_i \delta_{p_i}$, where $a_i, m_i$ and $p_i$ are as above.

\end{theorem}
The graph of the optimal penalty function given in Theorem \ref{thm:2tosses}, $f(p)=D(a_0,a_1,a_2;p)$, is shown in Figure 8. Observe that the interlacing property $p_0 < a_0 < p_1 < a_1 < p_2 < a_2 < p_3$ holds.

For the proof of Theorem \ref{thm:2tosses}  two technical lemmas are needed.
\begin{lemma} \label{lem:bounds}

The optimal choice $\{a_0, a_1, a_2 \}$ satisfies $0<a_0<0.2<0.4<a_1<0.6<0.8<a_2<1$. In addition, $a_1-a_0 < 0.4$ and $a_2-a_1 < 0.4$.

\end{lemma}

\begin{proof} It is easy to check that for the symmetric choice $\{a_0,a_1,a_2\}$, where $a_0 = .195\,,\,a_1 = .5\,,\,a_2 = 1-a_0 = .805$, we have that $\|D(a_0,a_1,a_2;p)\|_\infty=0.195\,<\,0.2$. Therefore, $\| f \|_\infty <0.2$. Since $f(0)=a_0$ and $f(1)=1-a_2$ we immediately obtain that $a_0<0.2$ and $0.8<a_2$.
On the other hand, if there exists $p\in[0,1]$ such that $|p-a_i | \geq 0.2$, $i=0,1,2$, then $\|f \|_\infty \geq \min_i |p-a_i | \geq 0.2$, which is a contradiction. Thus, $\min_i |p-a_i |<0.2 $ holds for all $p\in[0,1]$. This implies that $0<a_0$ and $a_2<1$ (otherwise, the Dirichlet Pigeonhole Principle implies there exists a $p$ such that $\min_i |p-a_i |\geq 0.2$). The fact that $\min_i |p-a_i |<0.2\,,\, p\in[0,1]\,,$ also implies that $a_1-a_0 < 0.4$ and $a_2-a_1 < 0.4$.

\noindent To derive that $0.4<a_1<0.6$ we need only use the values $p=0.4$ and $p=0.6$ to determine that from $\min_i |0.4-a_i |<0.2$ we must have $|0.4-a_1|<0.2$, or $a_1<0.6$, and from $\min_i |0.6-a_i |<0.2$ we must have $|0.6-a_1|<0.2$, or $0.4<a_1$.
\end{proof}

\begin{remark}\label{symchoice}
While the ``test''values used in the above proof might seem, at first glance, quite arbitrary, their usefulness can be easily shown by a simple numerical experimentation. In particular, for such a symmetric choice with $0<a_0\,,\,a_1 = .5$ and $a_2 = 1-a_0 < 1$, we see that the restrictions of the penalty function $f$ to the ``end'' subintervals $[0,a_0]$ and $[1-a_0,1]$ are straight lines whose respective maximum values (attained at the endpoints $0$ and $1$) are given by $a_0$, and the restrictions to the ``central'' intervals $[a_0,.5]$ and $[.5,1-a_0]$, are concave functions. This fact will be used again in the proof of Theorem \ref{thm:2tosses} below.
\end{remark}

Now, as a consequence of Lemma \ref{lem:bounds} we have the following result, which completes the previous Lemma \ref{lem:maxima}

\begin{lemma}\label{lem:01}
For an optimal choice $\{a_0, a_1, a_2 \}\,,$ we have that $\{0,1\} \subset M(f)\,.$
\end{lemma}

\begin{proof} Indeed, consider the restriction of $f(p)$ to $[0,a_0]$. Then,
$$f(p)=(a_0-2a_1+a_2)p^2 -(1+2a_0-2a_1)p+a_0.$$
Since the global maximum is attained on $[0,a_0]$, and it is not at $a_0$,
the only possibility of it not being at $p=0$ is when $a_0-2a_1+a_2<0$ and the $x$-coordinate of the parabola's vertex $(1/2+a_0-a_1)/( a_0-2a_1+a_2)>0$, or $1/2+a_0-a_1<0$.
This implies that $a_1-a_0>1/2$, which, as shown in Lemma \ref{lem:bounds}, is impossible if $f$ is the optimal solution of \eqref{PlayII}. Therefore, we derive that $0\in M(f)$.
In a similar fashion, one gets $1\in M(f)$.
\end{proof}

\vspace{3mm}
\noindent {\bf Proof of Theorem \ref{thm:2tosses}}. From Lemmas 2.3-2.4, Theorem 3.1 and Lemma 4.4, we already know that

\[ d\mu(p)=\beta_0\delta_0+\beta_1\delta_{p}+\beta_2 \delta_{q} +\beta_3 \delta_1 , \]
for some $\beta_i\geq 0$, $\beta_0+\beta_1+\beta_2+\beta_3=1$, and $p \in (a_0,a_1)$, $q\in (a_1,a_2)$. Since we are in the case of Nash-equilibrium, the quantity
\begin{eqnarray*} &&E(a_0,a_1, a_2;\mu) = \mathcal{E}(a_0,a_1, a_2; p,q):=\int_0^1 f(r)\, d\mu(r)\\
&& \quad =\beta_0 f(0)+\beta_1 f(p)+\beta_2 f(q)+\beta_3 f(1) \\
&& \quad =
a_0\left[ \beta_0 -\beta_1(1-p)^2-\beta_2(1-q)^2\right] +a_1\left[2\beta_1 p(1-p)-2\beta_2 q(1-q)\right]\\
&& \quad \ +a_2\left[ \beta_1 p^2+\beta_2 q^2 -\beta_3 \right] +\beta_3-\beta_1 p+\beta_2 q +2\beta_1 p(1-p)^2-2\beta_2 q^3.
\end{eqnarray*}
is a global minimum (w.r.t. $\{a_0, a_1, a_2\} \in (0,1)$), which implies that \[\frac{\partial \mathcal{E}}{\partial a_i}=0. \quad i=0,1,2.\]
Therefore, the coefficients in front of $a_0, a_1$, and $a_2$ vanish for the given choice of $p,\ q$, and $\beta_i$, $i=0,1,2,3$. Hence, the optimization problem \eqref{PlayII} becomes a constrained minimization problem
\begin{equation}\label{CMP}
\begin{split}
& Maximize \quad \beta_3-\beta_1 p+\beta_2 q +2\beta_1 p(1-p)^2-2\beta_2 q^3\\
& \\
&Subject \ to\ \left \{ \begin{array}{rcrcrcrcr}
\beta_0&+&\beta_1&+&\beta_2&+&\beta_3&=&1\\
\beta_0&-&(1-p)^2\beta_1&-&(1-q)^2\beta_2&\ &\ &=&0\\
\ &\ &2p(1-p)\beta_1& -&2 q(1-q)\beta_2 &\ &\ &=&0\\
\ &\ &p^2 \beta_1&+&q^2 \beta_2& -&\beta_3&=&0\\
\ & \ &p,q\in [0,1], &\ &\beta_i  \geq 0.
\end{array}\right.
\end{split}
\end{equation}
Eliminating  $\beta_0$ and $\beta_3$ we reduce \eqref{CMP} to
\begin{equation}\label{CMP2}
\begin{split}
& Maximize \quad \beta_1 p(1-p)(1-2p)+\beta_2 q(1-q)(1+2q)\\
& \\
&Subject \ to\ \left \{ \begin{array}{rcrcr}
2[1-p(1-p)]\beta_1&+&2[1-q(1-q)]\beta_2&=&1\\
p(1-p)\beta_1& -& q(1-q)\beta_2 &=&0\\
p,q\in [0,1], \beta_i  \geq 0.
\end{array}\right.
\end{split}
\end{equation}
Further, eliminating $\beta_2$ from \eqref{CMP2} we derive
\begin{equation}\label{CMP3}
\begin{split}
& Maximize \quad 2\beta_1 p(1-p)(1-p+q)\\
& \\
&Subject \ to \quad
2\beta_1 p(1-p) \left[ \frac{1}{p(1-p)}+\frac{1}{q(1-q)}-2 \right]=1, \quad p,q,\beta_1 \in [0,1].
\end{split}
\end{equation}
Substituting $2\beta_1 p(1-p)$ and denoting $x=1-p$, $y=q$, we obtain the minimization problem

\begin{equation}\label{CMP4} Minimize \quad \frac{\displaystyle{\frac{1}{x(1-x)}+\frac{1}{y(1-y)}-2}}{x+y},\quad x,y \in [0,1].\end{equation}

\noindent Since $1/[x(1-x)]$ is a convex function, it is easy to see that if $x+y$ is kept constant, the minimum in \eqref{CMP4} is attained when $x=y$, which implies that $q=1-p$ (and subsequently $\beta_3=\beta_0$ and $\beta_2=\beta_1$) is a necessary condition for a Nash-equilibrium selection. Moreover, assuming $x=y$, we obtain that $x$ must minimize the function
\[ g(x)=\frac{1-x+x^2}{x^2(1-x)}=\frac{1}{x^2}+\frac{1}{1-x}.\]
Differentiating, we see that $g'(x)=0$ has only one solution in $[0,1]$, which satisfies $x^3=2(1-x)^2$, or $(1-p)^3=2p^2$, so $p=p_1$, where $p_1$ is given in \eqref{p1}. It now follows easily that $\beta_0=m_0$ and $\beta_1=m_1\,,$ where $m_0$ and $m_1$ are given in \eqref{weights}.

Therefore, we have proven that if $p = p_1$ and $q = 1-p_1$, with $0 < a_0 < p_1 < a_1 < q=1-p_1 < a_2 < 1$ as above, then $E(a_0,a_1, a_2;\mu)$ does not depend on $a_0, a_1, a_2$, or,
$$E(a_0,a_1, a_2;\mu) \equiv E(p_1) = \frac{2p_1(1-p_1)^2}{p_1^2+(1-p_1)^2+1}  \approx 0.1916\,.$$

\noindent Keeping in mind the results of Theorem \ref{thm:equimax}, it follows that the optimal strategy for Player II, is given by: $a_0= E(p_1) \approx .1916...$, $a_1=0.5, a_2=1-a_0 \approx .8084$.
The graph of $f(p)=D(a_0,a_1,a_2;p)$ is shown in Figure \ref{2tossplot}.

\begin{figure}[h]
\centering\includegraphics*[width=3.0in]{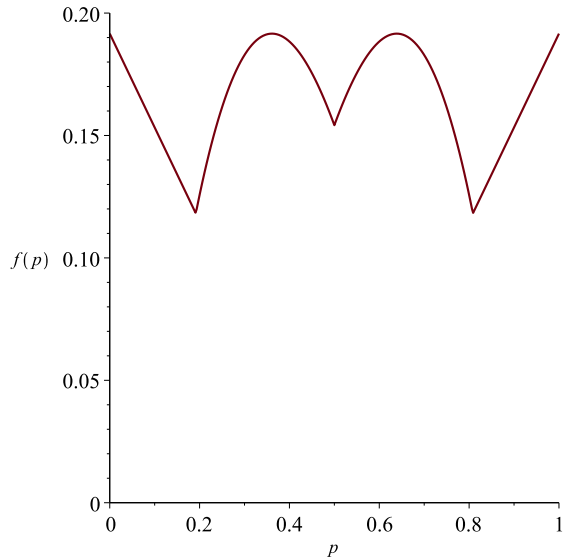}
\caption{Graph of $f(p)$ for $n=2$ tosses}
\label{2tossplot}
\end{figure}

\noindent Indeed, $f(0)=f(1)=a_0$, and direct but cumbersome calculations show that $f(p_1)=a_0$ and $f'(p_1)=0$, and by symmetry $f(1-p_1)=a_0$, $f'(1-p_1)=0$.
One has $f'(0)=-f'(1) \approx -0.38$ and $f''(p_1)=f''(1-p_1) \approx -4.43$. Since $f$ is a continuous piecewise-polynomial function whose restrictions to $[0,a_0]$ and to $[a_n,1]$ are straight lines, while the restrictions to $[a_0,a_1]$ and $[a_1,a_2]$ are concave functions, it follows that the set of the absolute maxima of $f$, $M(f)$,  cannot contain more than $4$ points, which are precisely $0, p_1, 1-p_1, 1$. Thus, $\|f\|_{\infty} = a_0 = E(p_1) = 0.1916...$.

\noindent Therefore, the Nash-equilibrium (\ref{Nash generalized}) is established. $\quad\Box $

\begin{remark}\label{rem:bilingual}
The proof of Theorem \ref{thm:2tosses}, as well as Example 4.5 (the one-toss case), show that the Nash strategy pair for Players II and I may be seen as the pair consisting of, the set of absolute error minimax estimators, and the least favorable prior distribution for which the former become Bayes estimators. However, unlike the case when the squared error loss function is used (see Remark 4.6), in the above proof of Theorem \ref{thm:2tosses}, as well as in the discussion of Example 4.5, it was shown that for the current absolute error loss function, given the least favorable distribution $\mu$ there is no unique corresponding Bayes estimator. Indeed, we have seen that for $n=1$ or $2$ tosses and for $\mu = \sum_{k=0}^{n+1}\,m_{i,n}\,p_{i,n}$ given in Example 4.5 and Theorem \ref{thm:2tosses}, respectively, every configuration $\{a_0,\ldots,a_n\}$ satisfying the interlacing property $0 = p_0 \leq a_0 \leq p_1 \leq \ldots \leq a_n \leq p_{n+1} = 1$ is a Bayes estimator for $\mu$. Does this hold for $n>2$? And, furthermore, in spite of the lack of uniqueness of the Bayes estimator, is our Minimax estimator unique? These issues will be taken up again in the next section.

\end{remark}

\setcounter{equation}{0}
\setcounter{theorem}{0}
\section{Conclusions and further remarks}
\label{Conclusions}

In this paper, minimax type techniques from Approximation Theory have been applied to a classical problem in Probability: estimating the probability of a biased coin after a few tosses. We used the Minimax Estimation with absolute error loss function to solve the problem, characterizing the optimal solution and studying the asymptotics of the optimal estimators as the number of tosses tend to infinity. In addition, the method employed has been described within the framework of a non-cooperative (in particular, zero-sum) two-player game, where both players are allowed to make use of mixed strategies which, in turn, is closely related to the connection between Minimax and Bayes estimators in Point Estimation Theory. Our main results are Theorem \ref{thm:equimax}, where the optimal strategy choice for the second player is characterized by means of a property with a striking resemblance to a well-known problem in Polynomial Interpolation, and Theorem \ref{LimDistr}, with its Corollary \ref{cor:weakconv}, where the uniform limiting distribution of the optimal choices is established. Likewise, the result of Theorem \ref{thm:2tosses}, where the Nash-equilibrium for the case of $n=2$ tosses is uniquely solved, is also remarkable.

In view of the results obtained, some further remarks and questions are noteworthy and will be subject of further research.

\begin{itemize}

\item We have shown that the sup-norm of the penalty function for the optimal choice is clearly smaller than the one corresponding to the Maximum Likelihood (MLE) choice (i.e. $a_k = k/n\,,\,k=0,\ldots,n\,,$ see Figure \ref{FiveToss}), especially for small values of $n$. But it is interesting to observe that for a slight modification of the MLE, namely taking $a_k = (k+1)/(n+2)\,,\,k=0,\ldots,n$, the results are much closer to those corresponding to the optimal choice (see Figure \ref{FiveTossmod}, where the initial Figure \ref{FiveToss} has been augmented by adding the graph of the penalty function for the modified MLE). Indeed, this modified MLE (MMLE) is much easier to compute than the optimal choice (especially for big values of $n$) and seems to provide near optimal results. In other words, if we replace the optimal selection by this MMLE, a {\em near Nash-equilibrium} arises (also commonly referred to as an $\varepsilon$-Nash equilibrium). This is a well-known problem in Game Theory: since the optimal strategies are often difficult to compute, it is customary to look for easily computable approximations for which the deviation from equality in the ``pure'' Nash-equilibrium \eqref{Nash generalized} is small enough. From a Game Theory standpoint, the difference between pure and near (or $\varepsilon$-) Nash-equilibria consists in the fact that while in the "pure" setting no player has motivation to modify its strategy (corresponding to the optimal strategy pair), in the near equilibrium setting there exists a small incentive to do it. Of course, our version of the near Nash-equilibrium using the MMLE only deals with the second player's strategy (for more information about near Nash-equilibrium, see e.g. \cite{nearNash}).

    This difference between ``pure'' and near Nash-equilibrium also has a counterpart regarding the optimality of the nodes in the sense of the Lebesgue constant in the context of polynomial interpolation. Indeed, if the interval $[-1,1]$ is considered, it is well known that the so-called Extended Chebyshev nodes (that is, the zeros of the Chebyshev polynomial $T_{n+1}$ adjusted in such a way that the first and last zero fall on the endpoints of the interval) provide a near optimal choice of interpolating nodes (see \cite{deBoorH}--\cite{Brutman}; see also \cite{Smith} for a deeper discussion on near optimal choices of nodes).

\begin{figure}[h]
\centering\includegraphics*[width=3.0in]{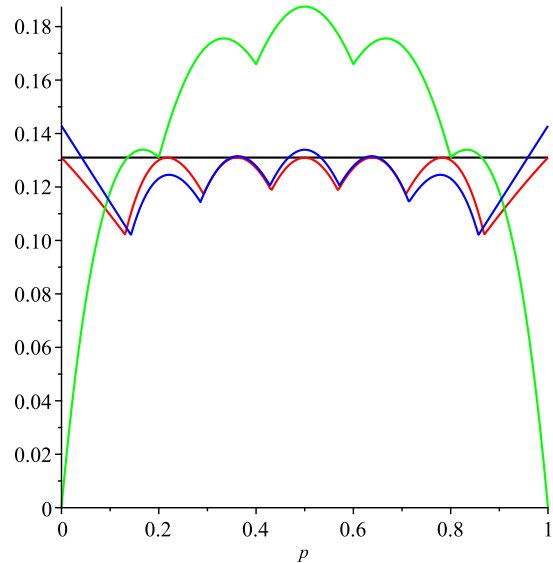}
\caption{Graph of $D(a_0,...,a_n;p)$ in the case of $n=5$ tosses: MLE  (green), MMLE (blue), optimal choice (AEME), (red), and the maximum value of the penalty function for the optimal choice (black) }
\label{FiveTossmod}
\end{figure}

\item However, Theorem \ref{LimDistr} shows that, as $n$ increases, the optimal choice, as well as the MLE and MMLE ones, or any other ``acceptable'' choice (such as, for instance, the SEME, given by (1.3)), all approach the same limiting distribution, in light of Remark \ref{rem:universality}. In other words, the advantage of using the optimal strategy over the MLE is worthwhile only for a small, or not very large, number of tosses.

\item In the solution of the Nash-equilibrium for the case of $n=2$ tosses we found that the interlacing property $0=p_0<a_0<p_1<a_1<p_2<a_2<p_3=1$ was satisfied. But what about the asymptotic distribution of the atomic measures $\displaystyle \sum_{j=0}^{n+1}\,m_{j,n}\,\delta_{p_{j,n}}\,$, with $\displaystyle \sum_{j=0}^{n+1}\,m_{j,n} = 1\,,$ which give the optimal strategy of the first player for each $n$ (provided they exist)? Of course, the same question may be posed using the language of Point Estimation Theory, regarding the pair formed by the Minimax estimator and its corresponding least favorable prior distribution.

\item As for the equimax property for the optimal choice of Player II (Minimax Estimation) given in Theorem \ref{thm:equimax}, it is necessary to point out a couple of pending questions about the uniqueness of the solution. First, is there a unique optimal configuration $\{a_0,\ldots,a_n\}$ for each $n$? And secondly, Theorem \ref{thm:equimax} proved that the intersection of the set of maxima $M(f)$ with each subinterval $(a_i,a_{i+1})$ corresponding to an optimal choice is nonempty. But, is there just a single absolute maximum on each subinterval? It was only established for the case of $n=2$ tosses and numerical results seem to confirm that it also holds for larger values of $n$.

\section*{Acknowledgement} The proof of Theorem \ref{LimDistr} is inspired from an argument presented by Vilmos Totik to the third author, for which the authors are very grateful. In addition, the authors wish to thank Prof. Jos\'{e} Luis Fern\'{a}ndez P\'{e}rez (Universidad Aut\'{o}noma, Madrid) for helping us to suitably place the present work within the
Point Estimation Theory literature.

\end{itemize}

\end{document}